\documentclass[11pt]{amsart}
\usepackage[textwidth=13.3cm,textheight=22.3cm,bindingoffset=4mm]{geometry}
\usepackage{graphicx}
\usepackage{appendix}
\usepackage{mdwlist}
\usepackage[colorlinks]{hyperref}
\usepackage[mathscr,mathcal]{euscript}

\newcommand{\apref}[3]{\hyperref[#2]{#1\ref*{#2}#3}}

\usepackage{enumerate}
\usepackage[latin1]{inputenc}
\usepackage{dsfont}
\usepackage{amssymb,amsthm,amsmath}

\input{xy}
\xyoption{all}

\usepackage{mathrsfs}

\theoremstyle{plain}
\newtheorem{prop}{Proposition}[section]
\newtheorem{lemma}[prop]{Lemma}

\newtheorem{thm}[prop]{Theorem}
\newtheorem*{thmnn}{Theorem}

\theoremstyle{definition}

\newtheorem{conj}[prop]{Conjecture}

\theoremstyle{remark}

\newcommand{\homsp}{\mathcal X}

\DeclareMathOperator{\diam}{diam}
\newcommand{\compact}{\mathcal K}

\DeclareMathOperator{\End}{End}

\DeclareMathOperator{\SL}{SL}

\DeclareMathOperator{\Ad}{Ad}

\DeclareMathOperator{\height}{ht}

\newcommand\N{\mathbb{N}}

\newcommand\R{\mathbb{R}}
\newcommand\Z{\mathbb{Z}}

\newcommand{\mc}[1]{\mathcal #1}
\newcommand{\mf}[1]{\mathfrak #1}

\newcommand{\mft}[2]{\mathfrak #1\mathfrak #2}
\newcommand{\wt}{\widetilde}

\newcommand{\eps}{\varepsilon}

\newcommand{\sceq}{\mathrel{\mathop:}=}
\newcommand{\seqc}{\mathrel{=\mkern-4.5mu{\mathop:}}}

\begin{document}

\title[Upper-semicontinuity and Hausdorff dimension]{Amount of failure of upper-semicontinuity of entropy in noncompact rank one situations, and Hausdorff dimension}
\author{S.\@ Kadyrov}
\author{A.\@ Pohl}
\address[AP]{Mathematisches Institut, Georg-August-Universit\"at G\"ottingen,  Bunsenstr. 3-5, 37073 G\"ottingen}
\email{pohl@uni-math.gwdg.de}
\address[SK]{School of Mathematics, University of Bristol, Bristol, UK}
\email{shirali.kadyrov@bristol.ac.uk}
\thanks{S.K.\@ acknowledges the support by the EPSRC. A.P.\@ acknowledges the support by the SNF (Grant 200021-127145) and by the ERC Starting Grant ANTHOS}
\begin{abstract}
Recently, Einsiedler and the authors provided a bound in terms of escape of mass for the amount by which upper-semicontinuity for metric entropy fails for diagonal flows on homogeneous spaces $\Gamma\backslash G$, where $G$ is any connected semisimple Lie group of real rank $1$ with finite center and $\Gamma$ is any nonuniform lattice in $G$. We show that this bound is sharp and apply the methods used to establish bounds for the Hausdorff dimension of the set of points which diverge on average.
\end{abstract}

\maketitle

\section{Introduction}

Let $G$ be a connected semisimple Lie group of $\R$-rank $1$ with finite center and $\Gamma$ a nonuniform lattice in $G$. Further let $a\in G\setminus\{1\}$ be chosen such that its adjoint action $\Ad_a$ on the Lie algebra $\mf g$ of $G$ is $\R$-diagonalizable. The element $a$ acts on the homogeneous space $\homsp\sceq \Gamma\backslash G$ by right multiplication, defining the (generator of the) discrete geodesic flow
\[
 T\colon \homsp\to \homsp,\ x\mapsto xa.
\]
The following relation between metric entropies of $T$ and escape of mass has been proven in \cite{EKP}. Here, $h_m(T)$ denotes the maximal entropy of $T$.

\begin{thmnn}
Let $(\mu_j)_{j\in\N}$ be a sequence of $T$-invariant probability measures on $\homsp$ which converges to the measure $\nu$ in the weak* topology. Then
\begin{equation}\label{mastereq}
\nu(\homsp) h_{\frac{\nu}{\nu(\homsp)}}(T) + \frac12 h_m(T) \cdot \left( 1- \nu(\homsp)\right) \geq \limsup_{j\to\infty} h_{\mu_j}(T),
\end{equation}
where it does not matter how we interprete $h_{\frac{\nu}{\nu(\homsp)}}(T)$ if $\nu(\homsp)=0$.
\end{thmnn}

Since $\Gamma$ is not cocompact, upper semi-continuity of metric entropy cannot be expected on $\homsp$. The theorem above shows that the amount by which it may fail is controlled by the escaping mass. In this formula, the factor $\tfrac12$ is significant: it shows that the amount of failure is only half as bad as it could be \textit{a priori} (which would be the factor $1$).

The first aim of this article is to show that the factor $\tfrac12$ is best possible. More precisely, we will establish the following theorem. 

\begin{thm}
\label{thm:entropylowerbound}
For any $c \in [\tfrac12h_m(T),h_m(T)]$, there exists a convergent sequence of $T$-invariant probability measures $(\mu_j)_{j\in\N}$ on $\homsp$ with $\lim_{j \to \infty} h_{\mu_j}(T)=c$ such that its weak* limit $\nu$ satisfies 
\[
\nu(\homsp) = \frac{2c}{h_m(T)}-1.
\]                                                                                                           
\end{thm}

For any such sequence $(\mu_j)$, equality holds in \eqref{mastereq} as well as
\[
 h_{\frac{\nu}{\nu(\homsp)}}(T) = h_m(T) \quad\text{for $\nu(\homsp)\not=0$}
\]
(and hence $\nu/\nu(\homsp)$ is the normalized Haar measure on $\homsp$).

The second aim of this article is to relate the factor $\tfrac12$ to the Hausdorff dimension of the set of points which diverge on average. We recall that a point $x \in \mc X$ is said to \emph{diverge on average} (with respect to $T$) if for any compact subset $\compact$ of $\homsp$ we have
\[
\lim_{n \to \infty}\frac{1}{n}\left|\left\{i \in \{0,1,\dots,n-1\} \mid T^i (x) \in \mc K \right\}\right|=0.
\]
It is said to be \textit{divergent} (with respect to $T$) if its forward trajectory under $T$ eventually leaves any compact subset. In other words, if for any compact subset $\compact$ of $\homsp$ we find $N\in\N$ such that for $n>N$ we have $T^nx\notin\compact$.

Obviously, each divergent point diverges on average. Let
\[
 U\sceq\{u \in G \mid a^{n} u a^{-n}\to 1 \text{ as } n \to \infty\}
\]
denote the unstable subgroup with respect to $a$. From \cite{DaniDivergent} and also from \cite{EKP} it follows that the Hausdorff dimension of the set of divergent points is $\dim G -\dim U$. However, for the set of averagely diverging points we prove that its Hausdorff dimension is strictly larger than $\dim G - \dim U$. Moreover, we also obtain an upper estimate showing that its dimension is strictly less than the full dimension. To state these results more detailed, let
\[
 \mc D \sceq \{ x\in\homsp \mid \text{$x$ diverges on average}\}.
\]
The Lie group $G$ has at most two positive roots, namely a short one, denoted $\alpha$, and the long one $2\alpha$. Let 
\[
 p_1 \sceq \dim \mf g_{\alpha} \quad\text{and}\quad p_2 \sceq \dim \mf g_{2\alpha}.
\]
The group $G$ has a single positive root if and only if it consists of isometries of a real hyperbolic space. In this case, we set $p_1=0$ or $p_2=0$ (both cases are possible and relevant, see Section~\ref{sec:prelims}).

\begin{thm}\label{thm:dim} For the Hausdorff dimension of $\mc D$ we have the estimates
\[
\dim G-\frac{1}{2}\dim U- \frac{p_2}2 \leq \dim \mc D \leq \dim G -\frac{1}{2}\dim U+\frac{p_1}{4}.
\]
\end{thm}

The proof of Theorem~\ref{thm:dim} shows that the factor $\tfrac12$ of $\dim U$ arises for the same reason as the factor $\tfrac12$ in \eqref{mastereq}. If $G$ consists of isometries of a real hyperbolic space, we obtain the following improvement. It is caused by the fact that in this case, the adjoint action of $a$ has a single eigenvalue of modulus greater than $1$. 

\begin{thm}\label{thm:dimsp}
Suppose that $G$ consists of isometries of a real hyperbolic space. Then 
\[
\dim\mc D = \dim G-\frac{1}{2}\dim U.
\]
\end{thm} 
 
Therefore, it seems natural to expect the following precise value for the Hausdorff dimension of $\mc D$.

\begin{conj}
If $G$ is any $\R$-rank $1$ connected semisimple Lie group with finite center, then $\dim_H\mc D = \dim G - \frac{1}{2}\dim U$.
\end{conj}

For the homogeneous spaces $\SL_{d+1}(\Z)\backslash \SL_{d+1}(\R)$, $d\geq 1$, and the action of a certain singular diagonal element of $\SL_{d+1}(\R)$, the analog of Theorem~\ref{thm:entropylowerbound} have been proven in \cite{Const}. For $d=2$, the Hausdorff dimension of the set of points which diverge on average in shown in \cite{Einsiedler_Kadyrov} to be $6+ 4/3$. 
\section{Preliminaries}\label{sec:prelims}

The Lie algebra $\mf g$ of the Lie group $G$ is the direct sum of a simple Lie algebra of rank $1$ and a compact one. The compact component does not have any influence on the dynamics considered here (cf.\@ \cite{EKP}). For this reason, we assume throughout that $\mf g$ is a simple Lie algebra of rank $1$ and, correspondingly, that $G$ is a connected simple Lie group of $\R$-rank $1$ with finite center. This allows us to work with a coordinate system for $G$ which is adapted to the dynamics, and $G$ can be realized as the isometry group of a Riemannian symmetric space of rank $1$ and noncompact type. For more background information on this coordinate system we refer to \cite{CDKR1,CDKR2}.

\textbf{Coordinate system.} Let $A$ be the maximal one-parameter subgroup of $G$ of diagonalizable elements which contains $a$, the chosen generator for the discrete geodesic flow $T$. Then there exists a group homomorphism $\alpha \colon A \to (\R_{>0},\cdot)$ such that $\alpha(a)>1$ and $\mf g$ decomposes into the direct sum
\begin{equation}\label{rootspaces}
 \mf g = \mf g_{-2} \oplus \mf g_{-1} \oplus \mf c \oplus \mf g_1 \oplus \mf g_2,
\end{equation}
where
\[
 \mf g_j \sceq \left\{ X \in \mf g \left\vert\ \forall\, \wt a\in A\colon \Ad_{\wt a}X = \alpha(\wt a)^{\frac{j}{2}}X \right.\right\},\quad j\in \{\pm 1,\pm 2\},
\]
and $\mf c$ is the Lie algebra of the centralizer $C=C_A(G)$ of $A$ in $G$. The homomorphism $\alpha$ is the square root of the ``group analog'' of the root $\alpha$ in the Introduction. If $\mf g$ is not isomorphic to $\mft so(1,n)$, $n\in \N$, the decomposition \eqref{rootspaces} is the restricted root space decomposition of $\mf g$. If $\mf g$ is isomorphic to $\mft so(1,n)$ for some $n\in \N$ (which is equivalent to say that $G$ consists of isometries of a real hyperbolic space), either $\mf g_1$ or $\mf g_2$ is trivial. In this case, both
\[
 \mf g = \mf g_{-1} \oplus \mf c \oplus \mf g_1 \quad\text{and}\quad \mf g=\mf g_{-2} \oplus \mf c \oplus \mf g_2
\]
are restricted root space decompositions of $\mf g$. The first one corresponds to the Cayley-Klein models of real hyperbolic spaces, the second one to the Poincar\'e models (see \cite{CDKR1,CDKR2}). In any case, let $\mf n \sceq \mf g_2 \oplus \mf g_1$ and let $N$ be the connected, simply connected Lie subgroup of $G$ with Lie algebra $\mf n$. Further pick a maximal compact subgroup $K$ of $G$ such that 
\[
 N\times A\times K \to G,\quad (n,\wt a, k) \mapsto n\wt ak \qquad\text{(Iwasawa decomposition)}
\]
is a diffeomorphism, and let 
\[
 M\sceq K\cap C.
\]
The semidirect product $NA$ is parametrized by
\[
 \R_{>0}\times \mf g_2 \times \mf g_1 \to NA, \quad (s,Z,X) \mapsto \exp(Z+X)\cdot a_s
\]
with $\alpha(a_s) = s$, $a_s\in A$. Let $\theta$ be a Cartan involution of $\mf g$ such that the Lie algebra $\mf k$ of $K$ is its $1$-eigenspace, and let $B$ denote the Killing form. Further let
\[
 p_1 \sceq \dim \mf g_1 \quad\text{and}\quad p_2 \sceq \dim \mf g_2.
\]
On $\mf n$ we define an inner product via
\[
 \langle X,Y\rangle \sceq -\frac{1}{p_1 + 4p_2} B(X,\theta Y) \qquad\text{for $X,Y\in\mf n$.}
\]
This specific normalization yields that the Lie algebra $[\cdot,\cdot]$ of $\mf g$, even though it is indefinite, satisfies the 
Cauchy-Schwarz inequality
\[
 |[X,Y]| \leq |X| |Y|
\]
for $X,Y\in\mf n$ (see \cite{Pohl_isofunddom}). We may identify $G/K \cong NA \cong \R_{>0} \times \mf g_2 \times \mf g_1$ with the space
\[
 D \sceq \left\{ (t,Z,X) \in \R \times \mf g_2 \times g_1 \left\vert\ t>\frac14|X|^2\right.\right\}
\]
via
\[
 \R_{>0} \times \mf g_2 \times \mf g_1 \to D,\quad (t,Z,X) \mapsto (t+\tfrac14|X|^2, Z, X).
\]
With the linear map $J\colon \mf g_2 \to \End(\mf g_1), Z\mapsto J_Z$,
\[
 \langle J_ZX,Y\rangle \sceq \langle Z, [X,Y]\rangle \quad\text{for all $X,Y\in \mf g_1$,}
\]
the geodesic inversion $\sigma$ of $D$ at the origin $(1,0,0)$ is given by (see \cite{CDKR2})
\begin{equation}\label{geodinv}
 \sigma(t,Z,X) = \frac{1}{t^2 + |Z|^2} \big(t, -Z, (-t+J_Z)X\big).
\end{equation}
We shall identify $\sigma$ with the element in $K$ with acts as in \eqref{geodinv}. Then $G$ has the Bruhat decomposition
\begin{equation}\label{bruhat}
  G = NAM \cup NAM\sigma N.
\end{equation}
To modify this Bruhat decomposition into one which is tailored to the dynamics on $\homsp$, we note the following result on fundamental domains of Siegel domain type. For $s>0$ let
\[
 A_s \sceq \{ a_t \in A \mid t>s\},
\]
and for any compact subset $\eta$ of $N$ define the Siegel set 
\[
 \Omega(s,\eta) \sceq \eta A_s K.
\]

\begin{prop}[Theorem~0.6 and 0.7 in \cite{Garland_Raghunathan}]\label{funddomGR}
There exists $s_0 > 0$, a compact subset $\eta_0$ of $N$ and a finite subset $\Xi$ of $G$ such that 
\begin{enumerate}[{\rm (i)}]
\item\label{funddomGRi} $G=\Gamma\Xi\Omega(s_0,\eta_0)$,
\item for all $\xi\in\Xi$, the group $\Gamma\cap \xi N \xi^{-1}$ is a cocompact lattice in $\xi N \xi^{-1}$,
\item for all compact subsets $\eta$ of $N$ the set
\[
 \{ \gamma \in \Gamma \mid \gamma\Xi\Omega(s_0,\eta) \cap \Omega(s_0,\eta) \not=\emptyset \}
\]
is finite,
\item\label{funddomGRiv} for each compact subset $\eta$ of $N$ containing $\eta_0$, there exists $s_1>s_0$ such that for all $\xi_1,\xi_2\in\Xi$ and all $\gamma\in\Gamma$ with $\gamma\xi_1 \Omega(s_0,\eta) \cap \xi_2\Omega(s_1,\eta) \not= \emptyset$ we have $\xi_1 = \xi_2$ and $\gamma\in \xi_1 NM \xi_1^{-1}$.
\end{enumerate}
\end{prop}

Throughout we fix a choice for $\eta_0$, $s_1$ (with $\eta=\eta_0$) and $\Xi$. The elements of $\Xi$ are representatives for the cusps of $\homsp$ (and will also be called cusps). Note that $U=\sigma N \sigma$. Multiplying \eqref{bruhat} with $\xi\in\Xi$ from the left and $\sigma$ from the right yields
\[
 G = \xi NAM\sigma \cup \xi NAMU.
\]
We may assume throughout that $a$ is chosen such that 
\[
 \alpha(a) = e,  \qquad (e=\exp(1))
\]
letting $T$ result in the time-one geodesic flow. By scaling, the statements of Theorem~\ref{thm:entropylowerbound}-\ref{thm:dimsp} are valid for a generic $a$ if proven in this particular case.
The subgroup $U$ is just the unstable subgroup with respect to $a$, and the conjugation of $\sigma(1,Z,X)\sigma \in U$ by $a$ is given by
\[
 a^{-k}\sigma (1,Z,X)\sigma a^k = \sigma (1, e^{-k}Z, e^{-k/2}X)\sigma \qquad (k\in\Z).
\]

\textbf{Maximal entropy.} The maximal metric entropy of the time-one geodesic flow $T$ is 
\[
 h_m(T) = \frac{p_1}{2} + p_2.
\]
It is uniquely realized by the normalized Haar measure on $\homsp$, which we denote by $m$.

\textbf{The height function and an improved choice of $s_1$.} In the following we recall the definition of the height function on $\homsp$ from \cite{EKP} and its significant properties. For any $\xi\in \Xi$ consider the $\xi$-Iwasawa decomposition $G=\xi NAK$. For $g\in G$ let $s=s_\xi(g) > 0$ be such that $g=\xi n a_s k$ for some $n\in N$, $k\in K$. For $x\in \homsp$, its $\xi$-height is
\[
 \height_\xi(x) = \sup\{ s_{\xi}(g) \mid \Gamma g = x\}.
\]
Its height is 
\[
 \height(x) = \max\{ \height_\xi(x) \mid \xi \in \Xi\}.
\]
For $s>0$ we set
\[
\homsp_{<s}=\{x \in \mc X : \height(x) <s\} \quad\text{and}\quad \homsp_{\geq s}=\{x \in \mc X : \height(x) \ge s\}.
\]
The constant $s_1$ in Proposition~\ref{funddomGR} can be chosen such that 
\begin{enumerate}[(i)]
\item if for $x\in\homsp$ and $\xi\in\Xi$, we have $\height_\xi(x)>s_1$, then $\height(x) = \height_\xi(x)$, 
\item if for $x\in\homsp$, we have $\height(x)>s_1$ and $\height(x)>\height(xa)$, then the $T$-orbit of $x$ strictly descends below height $s_1$ before it can rise again. This means that there exists $n\in\N$ such that for $j=0,\ldots, n-1$, we have $\height(xa^j)>\height(xa^{j+1})$ and $\height(xa^n)\leq s_1$, and
\item if $x\in \homsp$ and $\height_\xi(x)>s_1$ for some $\xi\in\Xi$, then there is (at least one) element $g=\xi na_rmu \in \xi NAMU$ or $g=\xi na_rm\sigma\in \xi NAM\sigma$ which realizes $\height_\xi(x)$. That is, $x=\Gamma g$ and $\height_\xi(x) = s_\xi(g)$. The components $a_r$ and $u$ do not depend on the choice of $g$.
\end{enumerate}
We suppose from now on that $s_1$ satisfies these properties.

For points $x\in\homsp$ which are high in some cusp, we have the following explicit formulas for the calculation of the height of the initial part of its orbit.

\begin{prop}[\cite{EKP}]
Let $x\in\homsp$, $\xi\in\Xi$ and suppose that $\height_\xi(xa^k)>s_1$ for all $k\in\{0,\ldots, n\}$. 
\begin{enumerate}[{\rm (i)}]
\item If $\height_\xi(x)$ is realized by $g=\xi n a_r m \sigma \in \xi NAM\sigma$, then 
\[
 \height_\xi(xa^k) = re^{-k}.
\]
\item If $\height_\xi(x)$ is realized by $g=\xi n a_r m u \in \xi NAMU$ with $u=\sigma (1, Z, X)\sigma$, then 
\[
 \height_\xi(xa^k) = r \frac{e^{-k}}{\left( e^{-k} + \frac14|X|^2\right)^2 + |Z|^2}.
\]
\end{enumerate}

\end{prop}

\textbf{Riemannian metric on $G$ and metric on $\homsp$.}
The isomorphism $\mf n=\mf g_2 \times \mf g_1 \to N$, $(Z,X)\mapsto \exp(Z+X)$, induces the inner product of $\mf n$ to $N$. Using the isomorphism $N\to U$, $n\mapsto \sigma n\sigma$, it gets further induced to $U$, and hence to $\overline{\mf n} \sceq \mf g_{-2}\times \mf g_{-1}$. 

We pick a left $G$-invariant Riemannian metric on $G$, which on the tangent space $T_1G \cong \mf g$ reproduces the inner products on $\mf n$ and $\overline{\mf n}$. Let $d_G$ denote the induced left-$G$-invariant metric on $G$. For $r>0$ let $B_r^G$, $B_r^U$, resp.\@ $B_r^{NAM}$ denote the $r$-balls in $G$, $U$, resp.\@ $NAM$ around $1\in G$.
We define
\begin{align*}
\lambda_0 & \sceq \min\{ |\lambda| \mid \text{$\lambda$ is an eigenvalue of $\Ad_{a}$ with $|\lambda|>1$}\}.
\end{align*}
Thus, 
\[
\lambda_0 =
\begin{cases}
e & \text{if $\mf g_1=\{0\}$ (and hence $G/K$ is a real hyperbolic space),}
\\
e^{1/2} & \text{otherwise.}
\end{cases}
\]
Then for any $L\geq 0$ we have
\[
 a^{L}B^U_r a^{-L} \subseteq B^U_{\lambda_0^{-L}r}
\]
or, in other words,
\[
d(ua^{-L}, va^{-L}) \leq \lambda_0^{-L} d(u,v) \leq d(u,v)
\]
for $u,v\in U$. Further
\[ 
 c \max\{ |Z|, |X| \} \leq d_G(1, \sigma (1,Z,X)\sigma)
\]
for some constant $c>0$ and all $u=\sigma (1,Z,X)\sigma \in U$. We avoid overly use of global constants, we may assume that $c=1$. 
The induced metric $d_{\homsp}$ on $\homsp$ is given by
\[
 d_\homsp (x,y) \sceq \min\{ d_G(g,h)\mid x=\Gamma g,\ y=\Gamma h\}.
\]
We usually omit the subscripts of $d_G$ and $d_\homsp$.

Finally, to shorten notation, we use 
\[
 [0,n] \sceq \{0,\ldots,n\}
\]
for $n\in\N$. The context will always clarify whether $[0,n]$ refers to this discrete interval or a standard interval in $\R$..
\section{Upper bound on Hausdorff dimension}\label{sec:upper}

Recall that
\[
 \mc D = \{ x\in \homsp \mid \text{$x$ diverges on average}\}.
\]

\begin{thm}\label{upper2}
The Hausdorff dimension of $\mc D$ is bounded from above by
\begin{equation}\label{case1}\tag{i}
 \dim \mc D \leq \dim G - \frac12\dim U + \frac{p_1}{4}.
\end{equation}
If $p_2=0$, then 
\begin{equation}\label{case2}\tag{ii}
 \dim \mc D \leq \dim G - \frac12\dim U.
\end{equation}
\end{thm}

The proof of this theorem builds on Lemma~\ref{measureest} below, which easily follows from the contraction rate of the unstable direction under the action of $a$.

\begin{lemma}\label{measureest}
Let $\mu$ be a probability measure on $\homsp$ of dimension at most $\beta$. Then, for any $r>0$, any $x\in\homsp$ and any $L\in\N$ we have
\[
\mu(xa^LB_r^{U}a^{-L}B_r^{NAM}) \leq c r^{\beta} e^{\left(\dim NAM + \frac{p_1}{2} - \beta\right)L}.
\]
If $p_2=0$, this bound can be improved to
\[
 \mu(xa^LB_r^{U}a^{-L}B_r^{NAM}) \leq c r^{\beta} e^{\left(\dim NAM - \beta\right)\frac{L}{2}}.
\]
Here, $c$ is a constant only depending on $\mu$.
\end{lemma}

\begin{proof}[Proof of Theorem~\ref{upper2}]
The claimed bound on the Hausdorff dimension of $\mc D$ follows as Theorem~1.4 and Corollary~1.5 in \cite{Einsiedler_Kadyrov}, using Lemmas~8.5 and 8.6 in \cite{EKP} as well as Lemma~\ref{measureest}.
\end{proof}

\section{Lower bound on Hausdorff dimension}\label{sec:lower}

In this section we prove the following lower bound on Hausdorff dimension:

\begin{thm}\label{thm:lower}
The Hausdorff dimension of the set of points in $\homsp$ which diverge on average is at least 
\[
 \dim G - \frac12 \dim U - \frac{p_2}2.
\]
\end{thm}

As a tool we use a lower estimate on the Hausdorff dimension of the limit set of strongly tree-like collections provided by \cite[\S4.1]{KleMar} (which goes back to \cite{Fal}, \cite{McM}, \cite{Urb}, and \cite{PesWei}). 

Let $U_0$ be a compact subset of $U$ and let $\lambda$ be the Lebesgue measure on $U$ (using the identification $U \cong \R^{p_2} \times \R^{p_1}$). A countable collection $\mathcal U$ of compact subsets of $U_0$ (a subset of the power set of $U_0$) is said to be \textit{strongly tree-like} if there exists a sequence $(\mc U_j)_{j\in\N_0}$ of finite nonempty collections on $U_0$ with $\mathcal U_0=\{U_0\}$ such that 
\[
 \mc U = \bigcup_{j\in\N_0} \mc U_j
\]
and
\begin{align}
\label{stl1}\forall\, j\in \mathbb N_0\ \forall\, A,B \in \mathcal U_j \text{ either $A=B$  or $\lambda(A\cap B)=0$},
\\
\forall\, j\in \mathbb N\ \forall\, B \in \mathcal U_j\ \exists\, A\in\mathcal U_{j-1} \text{ such that } B \subseteq A,
\\
d_j(\mathcal U):=\sup_{A\in \mathcal U_j} \diam(A) \to 0 \text{ as } j\to \infty.
\end{align}
Note that \eqref{stl1} implies $\lambda(A)>0$ for all $A\in\mathcal U$. For a strongly tree-like collection $\mathcal{U}$ with fixed sequence $(\mc U_j)_{j\in\N_0}$ we let 
\begin{equation}\label{fatUn}
{\bf U_j} \sceq \bigcup_{A \in \mathcal U_j} A\qquad \text{for any $j\in \mathbb N_0$.}
\end{equation}
Clearly, ${\bf U}_j \subset {\bf U}_{j-1}$ for any $j \in \mathbb N$. Further we call the nonempty set 
\begin{equation}\label{fatUinfty}
 {\bf U}_\infty \sceq \bigcap_{j\in\N_0} {\bf U}_j
\end{equation}
the \textit{limit set} of $\mc U$. For any subset $B$ of $U_0$ and any $j \in \mathbb N$ we define the {\it j-th stage density} of $B$ in $\mathcal U$ to be
\[
\delta_j(B,\mathcal U) \sceq  
\begin{cases} 
0  & \text{if $\lambda(B)=0$}
\\
\frac{\lambda({\bf U}_j\cap B)}{\lambda(B)} & \text{if $\lambda(B)>0$.} 
\end{cases}
\]
Note that $\delta_j(B, \mathcal U) \le 1$. Finally, for any $j\in \mathbb N_0$ we define the {\it j-th stage density} of $\mathcal U$ to be
\[
\Delta_j(\mathcal U) \sceq \inf_{B \in \mathcal U_j} \delta_{j+1}(B,\mathcal U).
\]

\begin{lemma}[\cite{KleMar}]\label{lem:Frostman}
For any strongly tree-like collection $\mathcal U$ of subsets of $U_0$ we have
\[
\dim_H({\bf U}_\infty) \ge \dim U -\limsup_{j \to \infty }\frac{\sum_{i=0}^{j-1}|\log(\Delta_i(\mathcal U))|}{|\log (d_j(\mathcal U))|}.
\]
\end{lemma}

\subsection{Construction of strongly tree-like collection}

We construct a strongly tree-like collection such that its limit set consists only of points which diverge on average. This construction proceeds in several steps. 

\begin{prop}
\label{prop:main}
Let $s>39s_1$ and $R\in\N$. Then there exists $x \in \mc X_{\leq s}$ such that for any $\eta$ in the interval $(0,\tfrac12)$ there exists a subset $E$ of $\overline B_{\eta e^{-R/4}}^{U}$ with $S=\lfloor  e^{R/2}\rfloor^{p_2}\lfloor e^{R/4}\rfloor^{ p_1}$ elements such that 
\begin{enumerate}[{\rm (i)}]
 \item\label{main1} for all $u\in E$, the points $x u$ and $T^{R}(xu)$ are contained in $\mc X_{\le s}$,
 \item for any two distinct elements $u,v\in E$ we have $d\big(T^{R}(u),T^{R}(v)\big) \geq  \eta$,
 \item\label{main3} for all $u\in E$ and all $k \in [0,R]$ we have $T^k(xu) \in \mc X_{>s/39}$.
\end{enumerate}
We may choose for $x$ any element $\Gamma g$ with
\[
 g\in\{\xi na_rm\sigma (1,Z_0,X_0)\sigma\mid n\in N, r\in I, m\in M\},
\]
where $\xi\in\Xi$ is any cusp, $I$ is a specific interval in $\R$ of positive length and $(1,Z_0,X_0)$ is a specific point in $N$, both being specified in the proof. Thus, the dimension of the set of possible $x$ is at least $\dim (NAM)$.
\end{prop}

\begin{proof}
Fix a cusp $\xi\in\Xi$ and pick an element $(Z_0,X_0)\in \mf g_2 \times \mf g_1$ with $|Z_0| = \frac32 e^{-R/2}$ and $|X_0| = \frac32 e^{-R/4}$. Define
\[
 g \sceq \xi n a_r m\sigma (1, Z_0,X_0) \sigma \quad\text{and}\quad x\sceq \Gamma g
\]
with $n\in N$, $m\in M$. Set 
\[
 B \sceq \{ (Z,X) \in \mf g_2 \times \mf g_1 \mid |Z|\leq \eta e^{-R/2},\ |X| \leq \eta e^{-R/4} \}.
\]
In the following we will estimate the height of $xa^k$, $k\in [0,R]$, and deduce an allowed range for $r$ such that $x$ satisfies \eqref{main3} and \eqref{main1}  for all elements in $\sigma B \sigma$. Since the height does not depend on $n$ and $m$, we omit these two elements. Let $(Z,X)\in B$.
Recall that 
\[
 g\sigma (1,Z,X) \sigma = \xi  a_r \sigma (1, Z_0 + Z + \frac12 [X_0,X], X_0 + X) \sigma.
\]
Then 
\begin{equation}\label{boundsX}
 e^{-R/4} < |X_0 + X| < 2 e^{-R/4}
\end{equation}
and, using $|[X_0,X]|\leq |X_0||X|$, 
\begin{equation}\label{boundsZ}
 \frac58 e^{-R/2}< \left| Z_0 + Z + \frac12 [X_0,X]\right| < 3 e^{-R/2}.
\end{equation}
Let $k\in [0,R]$. Recall that  
\begin{equation}\label{formulaheight}
\height_\xi\big(x\sigma (1,Z,X)\sigma a^k\big)  = r \cdot \frac{e^{-k}}{ \left(e^{-k} + \frac14|X_0+X|^2\right)^2 + \left| Z_0 + Z + \frac12[X_0,X]\right|^2}
\end{equation}
for sufficiently large $r$ (calculated below). Using the upper bounds in \eqref{boundsX} and \eqref{boundsZ} it follows that
\[
\height_\xi\big(x\sigma (1,Z,X)\sigma a^k\big) > \frac{r}{13}.
\]
Hence, \eqref{main3} is satisfied for $r>\frac{s}{3}$ (note that then $\frac{r}{13} > \frac{s}{39}>s_1$). Moreover, for these $r$, \cite[Proposition~5.5]{EKP} shows
\[
 \height\big(x\sigma (1,Z,X)\sigma a^n\big) = \height_\xi\big(x\sigma (1,Z,X)\sigma a^n\big).
\]
Using the lower bounds in \eqref{boundsX} and \eqref{boundsZ} we find
\[
 \height(x\sigma (1,Z,X)\sigma a^k) \leq \frac{r}{e^{-k} + \frac12 e^{-R/2} + \frac{25}{64}e^{k-R}}.
\]
For $r\leq \frac{25}{64}s$, this implies $\height(x\sigma(1,Z,X)\sigma a^k)\leq s$ for $k\in\{0,R\}$ and hence \eqref{main1}.

To define the set $E$, we may pick pairwise disjoint elements
\[
 (Z_i, X_j) \in B,\quad i=1,\ldots, \lfloor e^{R/2} \rfloor^{p_2},\ j=1,\ldots, \lfloor e^{R/4} \rfloor^{p_1}
\]
such that 
\[
 |Z_k - Z_\ell| \geq \eta e^{-R},\quad |X_k-X_\ell| \geq \eta e^{-R/2}
\]
whenever $k\not= \ell$. Define
\[
 E \sceq \{ \sigma (1, Z_i, X_j) \sigma \mid i=1,\ldots \lfloor e^{R/2} \rfloor^{p_2},\ j=1,\ldots, \lfloor e^{R/4} \rfloor^{p_1}\}.
\]
For any two distinct elements $\sigma(1,Z,X)\sigma, \sigma(1,Z',X')\sigma \in E$ we have
\begin{align*}
 d(\sigma (1,Z,X) \sigma a^R, &\sigma (1, Z', X')\sigma a^R) 
\\
&  \geq \max\left\{ \left|Z-Z'+\frac12[X,X']\right|e^R, |X-X'|e^{R/2}\right\}
\end{align*}
If $X\not=X'$, then 
\[
 d(\sigma (1,Z,X) \sigma a^R, \sigma (1, Z', X')\sigma a^R) \geq |X-X'|e^{R/2} \geq \eta.
\]
If $X=X'$, then
\[
 d(\sigma (1,Z,X) \sigma a^R, \sigma (1, Z', X')\sigma a^R) \geq |Z-Z'| e^{R} \geq \eta.
\]
This completes the proof.
\end{proof}

To simplify notation we use the following convention: Given a sequence $(S_k)_{k\in\N}$ of positive natural numbers, for any $n\in\N$  we let 
\[
 \mc S_n \sceq  \{ (i_1,\ldots, i_n) \mid 1\leq i_j \leq S_j,\ j=1,\ldots, n\} =  [1,S_1]\times \cdots \times [1,S_n]
\]
be the set of $n$-multi-indices with entries $1,\ldots, S_j$ in the $j$-th component.  If ${\bf i} = (i_1,\ldots, i_n) \in \mc S_n$ and $j\in [1,S_{n+1}]$, then we set
\[
 ({\bf i}, j) \sceq (i_1,\ldots, i_n, j) \quad\in \mc S_{n+1}.
\]
Finally we let
\[
 \mc S \sceq \bigcup_{n\in\N} \mc S_n.
\]
We let 
\[
B_\eps(\compact) \sceq \{ x\in\homsp\mid d(\compact, x) < \eps \}
\]
denote the $\eps$-thickening of the set $\compact \subseteq \homsp$.

\begin{thm}
\label{thm:main} 
Let $\compact$ be a compact subset of $\homsp$. For any $k\in\N$ let $R_k, S_k \in \N$ such that there exist a subset $E^{(k)} \subseteq U$ of cardinality $S_k$ and a point $x_k\in\compact$ such that for any $u\in E^{(k)}$ we have
\begin{equation} \label{eqn:inK}
x_k u, T^{R_k}(x_k u)  \in \compact.
\end{equation}
Then for any ${\bf i} \in \mc S$ there exists $g_{\bf i}\in U$ such that, if we define
\[
 E'_n \sceq \{ g_{\bf i} \mid {\bf i} \in \mc S_n\} \quad\text{for $n\in\N$,}
\]
the following properties are satisfied:
\begin{enumerate}[{\rm (i)}]
\item\label{Eni} $E'_1 = E^{(1)}$,
\item\label{Enii} for any $m\in\N$ there exists an enumeration of $E^{(m)}$ by $[1, S_m]$, say
\[
 E^{(m)} = \left\{ u_1^{(m)},\ldots, u_{S_m}^{(m)}\right\},
\]
and for any $\eta > 0$ there exists $R'= R'(\eta, \compact) \in \N$ (independent of the choice of the $g_{\bf i}$'s) such that with
\[
 F(k) \sceq \sum_{i=1}^{k-1} R_i + (k-1)R',\quad k\in\N,
\]
we have
\begin{equation}\label{eqn:n and n+1}
 d\big( T^{F(n) + R_n}g_{\bf i}, T^{F(n)+R_n} g_{({\bf i},j)}\big) < \eta
\end{equation}
for any $n\in\N$, ${\bf i}\in \mc S_n$, and $j\in [1,S_{n+1}]$, and
\begin{equation}\label{eqn:main}
 T^{F(k)}(x_1 g_{\bf i})\in x_k u_{i_k}^{(k)} B_{\eta/2}^{NAM}a^{R_k}B_{\eta/2}^{U}a^{-R_k}
\end{equation}
for any $n\in\N$, any ${\bf i} = (i_1,\ldots, i_n) \in \mc S_n$ and any $k\in [1,n]$.
\suspend{enumerate}
If, in addition, $\eta_0>0$ is an injectivity radius of $B_\eps(\compact)$ for some (fixed) $\eps > 0$, and
\[
 E^{(k)} \subseteq B^U_{\eta_0/4} \quad\text{for all $k\in\N$,}
\]
and
\[
 d\big( T^{R_k}u, T^{R_k}v) \geq \eta_0
\]
for any distinct $u,v\in E^{(k)}$, any $k\in\N$, and in \eqref{Enii} we have
\[
 \eta < \min\left\{\frac{\eta_0 (\lambda_0-1)}{4\lambda_0},\frac{\eps}2\right\}
\]
then 
\resume{enumerate}[{[{\rm (i)}]}]
\item\label{Eniii} for any $n\in\N$, the set $E'_n$ has the cardinality of $\mc S_n$, and 
\item\label{Eniv} for any $n\in\N$, any distinct ${\bf i, j} \in \mc S_n$ we have 
\[
\eta_0 > d(g_{\bf i}, g_{\bf j}) \quad\text{and}\quad 
 d\big( T^{F(n)+R_n}g_{\bf i}, T^{F(n)+R_n}g_{\bf j}\big) > \frac{\eta_0}2.
\]
\end{enumerate}
\end{thm}

The proof of Theorem~\ref{thm:main} is based on Lemmas~\ref{lem:shadow}-\ref{lem:joinshadow} below. Throughout these lemmas we let $\compact$ be a fixed compact subset of $\homsp$.

Recall that the group $UNAM$ is a neighborhood of $1\in G$. We fix $\eps_1 > 0$ such that $B_{\eps_1}^G \subseteq UNAM$. The Shadowing Lemma~\ref{lem:shadow} below uses the fact that the subgroups $NAM$ and $U$ intersect in the neutral element $1$ only. 

\begin{lemma}[Shadowing Lemma]
\label{lem:shadow}
There exists $c>0$ such that for any $\eps \in (0,\eps_1)$ and $x_-,x_+ \in \homsp$ with $d(x_{-},x_{+})< \eps$ there exist $u^+\in B_{c\eps}^{U}$ and $u \in B_{c\eps}^{NAM}$ such that
\begin{equation}
 \label{eqn:shadow+}x_- u^+=x_+u
 \end{equation}
\end{lemma}
\begin{proof}
There exists $g \in G$ with $d(g,1)<\eps$ such that $x_-g=x_+.$ Write $g=u^+u^{-1}$ with $u \in NAM$ and $u^+ \in U$. Then, $d(u^+,1)<c\eps$ and $d(u,1)<c\eps$ and $x_-u^+=x_+u.$ Now continuity of the decomposition, continuous dependence of $c$ on $u^+$ and $u$, and the bounded range for $\eps$ implies a uniform constant $c$.
\end{proof}

The compactness of $\compact$ and the topological mixing of $T$ imply the following lemma.

\begin{lemma}
\label{lem:join}
For any $\eta>0$ and any $\delta>0$ there exists $R'=R'(\delta,\compact,\eta)\in\N$ such that for any $z_-,z_+ \in B_\eta(\compact)$ and $\ell\ge R'$ there exists $z' \in \homsp$ such that $d(z',z_-)<\delta$ and $d(z_+,T^{\ell}(z'))< \delta$.
\end{lemma}

The proof of the following lemma is a combination of Lemmas~\ref{lem:shadow} and \ref{lem:join}.

\begin{lemma}
\label{lem:joinshadow}
Let $\eta > 0$ and let $z_-$ and $z_+$ be in $B_\eta(\compact)$. Let $c$ be as in the Shadowing Lemma~\ref{lem:shadow}. For any $\delta>0$ let $R'=R'(\delta,\compact,\eta)$ be as in Lemma~\ref{lem:join}. Then there exist $u^+ \in B_{c(c+2)\delta}^{U}$ and $u \in B_{c(c+2)\delta}^{NAM}$ such that
\[
T^{R'}(z_-u^+)=z_{+}u.
\]
\end{lemma}

\begin{proof}
We will throughout assume that $\delta< \frac{\eps_1}{c+1}$ to be able to apply the Shadowing Lemma~\ref{lem:shadow}. If the statement is proven for these small $\delta$, it holds \textit{a fortiori} for larger $\delta$. We first use Lemma~\ref{lem:join} to obtain $z' \in \homsp$ such that 
\begin{equation}
\label{eqn:join}
d(z',z_-)<\delta \quad\text{and}\quad d\big(z_+,T^{R'}(z')\big)< \delta.
\end{equation}
Now we apply Lemma~\ref{lem:shadow} with $x_-=z_-,x_+=z'$ and $\eps=\delta$ to obtain $u_1^+ \in B_{c \delta}^{U}$ and $u_1 \in B_{c\delta}^{NAM}$ such that
 \begin{equation}
 z_- u_1^+=z' u_1.
 \end{equation}
The distance between $T^{R'}(z_-u_1^+)$ and $z_+$ is bounded as follows:  
\begin{align*}\label{eqn:forward}
d\big(T^{R'}(z_- u_1^+),z_+\big)& = d\big(T^{R'}(z' u_1),z_+\big)
\\
& \leq d\big(T^{R'}(z'u_1),T^{R'}z'\big)+d\big(T^{R'}z',z_+\big) 
\\
& < (c+1)\delta.
\end{align*}
We again apply Lemma~\ref{lem:shadow}, this time for $x_-=T^{R'}(z_-u_1^+), x_+=z_+$ and $\eps=(c+1)\delta $ to obtain $u_2^+ \in B_{c(c+1)\delta}^{U}$ and $u \in B_{c(c+1)\delta}^{NAM}$ such that
\[
T^{R'}(z_-u_1^+) u_2^+=z_+ u.
\]
Now $T^{R'}(z_-u_1^+) u_2^+=T^{R'}\big(z_-(u_1^+a^{R'} u_2^+a^{-R'})\big)$. Setting $u^+ \sceq u_1^+(a^{R'} u_2^+a^{-R'})$ concludes the proof.
\end{proof}

\begin{proof}[Proof of Theorem~\ref{thm:main}]

We start by proving \eqref{Eni} and \eqref{Enii}. To that end let $\eta > 0$ be arbitrary and pick $c>0$ as in the Shadowing Lemma~\ref{lem:shadow}. Set $D_\eta\sceq B_\eta(\compact)$, 
\[
 \delta\sceq \frac{\eta}2 \cdot \frac{\lambda_0-1}{c(c+2)\lambda_0}
\]
and fix $R'$ with the properties as in Lemma~\ref{lem:join} applied for this $\delta$. Instead of proving \eqref{eqn:main} we will prove the stronger statement
\begin{equation}
\label{eqn:i_k}
T^{F(k)}(x_1 g_{\bf i})\in x_k u_{i_k}^{(k)}B_{c(c+2)\delta}^{NAM}a^{R_k}B_{r(n,k)}^{U}a^{-R_k}
\end{equation}
for any $n\in\N$, any ${\bf i} = (i_1,\ldots, i_n) \in \mc S_n$ and any $k \in [1,n]$ where 
\[
r(n,k):=c(c+2)\delta \sum_{i=0}^{n-k-1} \lambda_0^{-i}
\]
and $r(n,n) = 0$ by convention. Since $c(c+2)\delta <\eta/2$ and $r(n,k)<\eta/2$, this is indeed stronger than \eqref{eqn:main}. For the proof of \eqref{eqn:i_k} we precede by induction on $n$. As a by-product, we will prove \eqref{Eni} and \eqref{eqn:n and n+1}. 

For $n=1$ and $j\in [1, S_1]$ we set $g_i=u_i^{(1)}$. Then \eqref{Eni}  and \eqref{eqn:i_k} for $n=1$ are trivially satisfied. Suppose that for some $n\in\N$ we constructed the set $E'_n$ fulfilling \eqref{eqn:i_k}. We show how to construct $E'_{n+1}$ from $E'_n$ such that \eqref{eqn:i_k} is satisfied for $n+1$ and \eqref{eqn:n and n+1} for $n$.

Let ${\bf i} \in \mc S_n$ and $j\in [1, S_{n+1}]$. By inductive hypothesis 
\[
 T^{F(n)}(x_1g_{\bf i}) \in x_n u_{i_n}^{(n)}B_{\frac{\eta}2}^{NAM} a^{R_n} B_{\frac{\eta}2}^U a^{-R_n}.
\]
Thus,
\begin{align*}
T^{F(n)+R_n}(x_1g_{\bf i}) \in T^{R_n}(x_nu_{i_n}^{(n)}) a^{-R_n} B_{\frac{\eta}2}^{NAM} a^{R_n} B_{\frac{\eta}2}^U.
\end{align*}
From 
\[
 a^{-R_n}B_{\frac{\eta}2}^{NAM}a^{R_n}B_{\frac{\eta}2}^U \subseteq B_\eta^G
\]
and $T^{R_n}(x_nu_{i_n}^{(n)}) \in \mc K$, it follows that $T^{F(n)+R_n}(x_1g_{\bf i}) \in D_\eta$. Further, $x_{n+1}u_j^{(n+1)}\in \mc K \subseteq D_\eta$. We apply Lemma~\ref{lem:joinshadow} with
\[
 z_-\sceq T^{F(n)+R_n}(x_1 g_{\bf i}) \quad\text{and}\quad z_+\sceq x_{n+1}u_j^{(n+1)}
\]
to obtain $u_j^+\in B^U_{c(c+2)\delta}$ and $u_j\in B^{NAM}_{c(c+2)\delta}$ satisfying
\begin{equation}\label{eqn:z2}
 x_1g_{\bf i}a^{F(n)+R_n}u_j^+ a^{R'} = T^{R'}(z_-u_j^+) = z_+ u_j = x_{n+1}u_j^{(n+1)}u_j.
\end{equation}
We define
\[
 g_{({\bf i},j)} \sceq g_{\bf i} a^{F(n)+R_n} u_j^+ a^{-F(n)-R_n} \quad \in U
\]
and
\[
 E'_{n+1} \sceq \{ g_{({\bf i},j)} \mid {\bf i}\in \mc S_n,\ j\in [1,S_{n+1}]\}.
\]
Clearly, 
\[
d\big( T^{F(n)+R_n}(g_{\bf i}), T^{F(n)+R_n}(g_{({\bf i},j)})\big) = d(1,u_j^+) < \frac{\eta}2,
\]
which proves \eqref{eqn:n and n+1} for $n$. 

We will now show \eqref{eqn:i_k} for $n+1$. Suppose first that $k=n+1$. From the definition of $F(n+1)$ and \eqref{eqn:z2} it immediately follows that 
\[
T^{F(n+1)}(x_1 g_{({\bf i},j)})  \in x_{n+1}u_j^{(n+1)} B_{c(c+2)\delta}^{NAM}.
\]
Suppose now that $k\in [1,n]$. Then 
\begin{align*}
T^{F(k)}(x_1 g_{({\bf i},j)}) & = x_1 g_{\bf i} a^{F(n)+R_n} u_j^+ a^{F(k)-F(n)-R_n}
\\
& = T^{F(k)}(x_1g_{\bf i}) a^{-F(k)+F(n)+R_n} u_j^+ a^{F(k)-F(n)-R_n}
\\
& \in T^{F(k)}(x_1 g_{\bf i}) a^{-F(k)+F(n)+R_n} B^U_{c(c+2)\delta} a^{F(k)-F(n)-R_n}.
\end{align*}
From the inductive hypothesis we have
\[
T^{F(k)}(x_1 g_{\bf i}) \in x_ku_{i_k}^{(k)} B^{NAM}_{c(c+2)\delta} a^{R_k} B^U_{r(n,k)} a^{-R_k}.
\]
Therefore
\begin{multline}\label{eqn:inductivestep}
T^{F(k)}(x_1 g_{({\bf i},j)}) \\
\in x_k u_{i_k}^{(k)} B_{c(c+2)\delta}^{NAM} a^{R_k} B_{r(n,k)}^{U}a^{-F(k)-R_k + F(n)+R_n} B_{c(c+2)\delta}^{U}a^{F(k)-F(n)-R_n}.
\end{multline}
If $k=n$, then $r(n,k)=0$. Hence \eqref{eqn:inductivestep} simplifies to
\begin{equation*}
T^{F(n)}(x_1 g_{({\bf i}, j)}) \in x_n u_{i_n}^{(n)} B_{c(c+2)\delta}^{NAM} a^{R_n} B_{c(c+2)\delta}^{U}a^{-R_n}.
\end{equation*}
If $k\in [1,n-1]$, then 
\[
 -F(k)-R_k + F(n) + R_n = \sum_{i=k+1}^n R_i + (n-k)R' \seqc p(k,n).
\]
Hence
\begin{align*}
a^{-F(k)-R_k + F(n) + R_n} B_{c(c+2)\delta}^{U}a^{F(k)+R_k-F(n)-R_n}&\subseteq B_{c(c+2)\delta \lambda_0^{-p(k,n)}}^{U}
\\ &
\subseteq B_{c(c+2)\delta \lambda_0^{-(n-k)}}^{U}.
\end{align*}
With $r(n,k)+c(c+2)\delta \lambda_0^{-(n-k)}=r(n+1,k)$ it now follows
\[
T^{F(k)}(x_1 g_{({\bf i},j)}) \in  x_k u_{i_k}^{(k)} B_{c(c+2)\delta}^{NAM} a^{R_k} B_{r(n+1,k)}^{U}a^{-R_k}.
\]
This completes the proof of \eqref{Enii}.

Since \eqref{Eniii} is an immediate consequence of \eqref{Eniv}, it remains to prove the two statements in \eqref{Eniv}. We start with the first one. Let ${\bf i} = (i_1,\ldots, i_n), {\bf j}=(j_1,\ldots, j_n) \in \mc S_n$. Then 
\[
 d(g_{\bf i}, g_{\bf j}) \leq d(g_{\bf i}, g_{i_1}) + d(g_{i_1},g_{j_1}) + d(g_{j_1},g_{\bf j}).
\]
Since $g_{i_1},g_{j_1}\in E^{(1)} \subseteq B^U_{\eta_0/4}$, we have $d(g_{i_1},g_{j_1}) < \eta_0/2$. To bound the other two terms, let $k\in [1,S_{n+1}]$. Then by \eqref{eqn:n and n+1} we have
\[
 d\big( T^{F(n)+R_n} g_{\bf i}, T^{F(n)+R_n} g_{({\bf i},k)}\big) < \eta.
\]
Therefore,
\[
 d(g_{\bf i}, g_{({\bf i},k)}) < \eta \lambda_0^{-F(n)-R_n}.
\]
Applying this observation iteratively, we obtain
\[
 d(g_{i_1}, g_{\bf i}) < \eta \sum_{j=1}^{n-1} \lambda_0^{-F(j)-R_j} < \eta \cdot\frac{1}{\lambda_0-1} < \frac{\eta_0}{4}.
\]
Thus, 
\[
 d(g_{\bf i}, g_{\bf j}) < \eta_0
\]
as claimed.

Finally, let ${\bf i}, {\bf j}\in\mc S_n$, ${\bf i}\not={\bf j}$. It remains to show that 
\begin{equation}\label{lastfact}
 d(T^{F(n)+R_n}g_{\bf i}, T^{F(n)+R_n}g_{\bf j}) > \frac{\eta_0}2.
\end{equation}
Suppose first that we find $k\in [1,n]$ such that 
\[
 d(g_{\bf i}a^{F(k)}, g_{\bf j}a^{F(k)}) \geq \eta_0.
\]
Since $F(k)-F(n)-R_n < 0$, the assumption
\[
 d(g_{\bf i}a^{F(n)+R_n}, g_{\bf j}a^{F(n)+R_n}) \leq \frac{\eta_0}{2}
\]
would result in 
\[
 d(g_{\bf i}a^{F(k)}, g_{\bf j}a^{F(k)}) \leq \frac{\eta_0}2.
\]
Therefore, in this case, \eqref{lastfact} is obviously satisfied.

To complete the proof pick $k\in [1,n]$ such that $i_k\not=j_k$ and suppose 
\[
 d(g_{\bf i}a^{F(k)}, g_{\bf j}a^{F(k)}) < \eta_0.
\]
Actually we may suppose $\leq \eta_0/2$, but $< \eta_0$ turns out to be sufficient. By \eqref{eqn:main} we find $u_i^-, u_j^- \in B_{\eta/2}^{NAM}$ and $u_i^+,u_j^+\in B_{\eta/2}^U$ such that
\begin{align*}
 T^{F(k)}(x_1g_{\bf i}) & = x_ku_{i_k}^{(k)}u_i^-a^{R_k}u_i^+a^{-R_k}
\intertext{and}
T^{F(k)}(x_1g_{\bf j}) & = x_ku_{j_k}^{(k)}u_j^-a^{R_k}u_j^+a^{-R_k}.
\end{align*}
Pick $h_0,h_k\in G$ such that $\Gamma h_0 = x_1$ and $x_k=x_1h_k$. Further let $\gamma\in\Gamma$ be such that 
\[
 \gamma h_0 g_{\bf i} a^{F(k)} = h_0 h_k u_{i_k}^{(k)} u_i^- a^{R_k} u_i^+ a^{-R_k}.
\]
We will show that 
\begin{equation}\label{samegamma}
 \gamma h_0 g_{\bf j} a^{F(k)} = h_0 h_k u_{j_k}^{(k)} u_j^- a^{R_k} u_j^+ a^{-R_k}
\end{equation}
(same $\gamma$!). To that end we note that 
\begin{align*}
& d\big(h_0h_k u_{i_k}^{(k)} u_i^- a^{R_k} u_i^+ a^{-R_k}, h_0h_k u_{j_k}^{(k)} u_j^- a^{R_k} u_j^+ a^{-R_k}\big)
\\
& \quad \leq d\big( u_{i_k}^{(k)} u_i^- a^{R_k} u_i^+ a^{-R_k}, u_{i_k}^{(k)}\big) + d\big(u_{i_k}^{(k)}, u_{j_k}^{(k)}\big) + d\big(u_{j_k}^{(k)}, u_{j_k}^{(k)}u_j^- a^{R_k} u_j^+ a^{-R_k}\big)
\\
& \quad < \eta + \frac{\eta_0}2 + \eta < \eta_0
\end{align*}
and
\[
 d\big(\gamma h_0 g_{\bf i} a^{F(k)}, \gamma h_0 g_{\bf j} a^{F(k)} \big) < \eta_0.
\]
Since $\eta_0$ is an injectivity radius of  $\partial_{B_\eps^G}\mc K$, now \eqref{samegamma} follows. Finally,
\begin{align*}
& d\big( g_{\bf i} a^{F(n)+R_n} , g_{\bf j} a^{F(n)+R_n} \big) 
\\
& \quad \geq d\big(g_{\bf i} a^{F(k)+R_k}, g_{\bf j} a^{F(k)+R_k}\big)
\\
& \quad =  d\big( u_{i_k}^{(k)} u_i^- a^{R_k} u_i^+, u_{j_k}^{(k)} u_j^- a^{R_k} u_j^+\big)
\\
& \quad \geq d\big(u_{i_k}^{(k)}a^{R_k}, u_{j_k}^{(k)}a^{R_k}\big) - d\big(u_{i_k}^{(k)} a^{R_k}, u_{i_k}^{(k)} u_i^- a^{R_k} u_i^+\big) 
\\ & \hphantom{\quad \geq d\big(u_{i_k}^{(k)}a^{R_k}, u_{j_k}^{(k)}a^{R_k}\big)}  - d\big(u_{j_k}^{(k)}a^{R_k}, u_{j_k}^{(k)}u_j^- a^{-R_k} u_j^+\big)
\\ 
& \quad\geq \eta_0 - 2\eta > \frac{\eta_0}2.
\end{align*}
This completes the proof.
\end{proof}

\textbf{Definition of strongly tree-like collection.} Fix $s_0>39s_1$ and set $\compact\sceq \homsp_{\leq s_0}$. Further fix an injectivity radius $\eta_0$ of some neighborhood of $\compact$ such that $\frac12>\eta_0>0$ and choose
\[
 \eta < \frac{\eta_0 (\lambda_0-1)}{4\lambda_0}
\]
so small that we may apply Theorem~\ref{thm:main}. For $k\in\N$ we set $\wt R_k\sceq k$ and 
\[
 \wt S_k \sceq \lfloor e^{k/2} \rfloor^{p_2} \cdot \lfloor e^{k/4} \rfloor^{p_1}.
\]
For any $k\in\N$ we apply Proposition~\ref{prop:main} with $\wt R_k$, $\wt S_k$, $s_0$ and $\eta_0$ to get a point $x_k\in\compact$ and a subset $\wt E^{(k)}\subseteq \overline B^U_{\eta_0 e^{-k/4}}$ with the properties of this proposition. For $k\geq k_0\sceq \lceil 4\log 4 \rceil$ we have $\wt E^{(k)} \subseteq B^U_{\eta_0/4}$. We set $E^{(k)} \sceq \wt E^{(k+k_0-1)}$, $R_k\sceq \wt R_{k+k_0-1}$, $S_k\sceq \wt S_{k+k_0-1}$ for $k\in\N$ and  apply Theorem~\ref{thm:main} to these sequences to construct a sequence $(E'_n)_{n\in\N}$ of sets with the properties as in Theorem~\ref{thm:main}. For any $n\in\N$ we set
\[
 \mc U_n \sceq \left\{ ua^{F(n)+R_n} \overline{B}^U_{\eta_0/4}a^{-F(n)-R_n} \left\vert\  u\in E'_n \vphantom{a^{F(n)+R_n} \overline{B}^U_{\eta_0/4}a^{-F(n)-R_n}}\right.\right\}.
\]
Let 
\[
 U_0 \sceq \bigcup \mc U_1 = \bigcup_{u\in E'_1} ua^{k_0} \overline B^U_{\eta_0/4} a^{-k_0},
\]
which is a compact non-null subset of $U$, and let $\mc U_0\sceq \{ U_0 \}$. 
We claim that 
\[
 \mc U \sceq \bigcup_{n\in\N_0}\mc U_n 
\]
is a strongly tree-like collection on $U_0$. 
To that end let $n\in\N$. Suppose that $g,h\in E'_n$, $g\not=h$. By Theorem~\ref{thm:main} we have 
\[
 d\big(ga^{F(n)+R_n}, ha^{F(n)+R_n}\big) > \frac{\eta_0}2.
\]
Therefore
\[
 ga^{F(n)+R_n} \overline{B}^U_{\eta_0/4} \cap h a^{F(n)+R_n} \overline{B}^U_{\eta_0/4} = \emptyset,
\]
and hence
\[
 ga^{F(n)+R_n} \overline{B}^U_{\eta_0/4} a^{-F(n)-R_n} \cap ha^{F(n)+R_n}\overline{B}^U_{\eta_0/4}a^{-F(n)-R_n} = \emptyset.
\]
This shows \eqref{stl1} (and even a stronger disjointness). Now let ${\bf i} \in \mc S_n$ and $j\in [1,S_{n+1}]$. We claim that 
\[
 g_{({\bf i},j)} a^{F(n+1)+R_{n+1}} \overline{B}^U_{\eta_0/4} a^{-F(n+1)-R_{n+1}} \subseteq g_{\bf i} a^{F(n)+R_n} \overline{B}^U_{\eta_0/4}a^{-F(n)-R_n},
\]
which is equivalent to
\begin{align}\label{nested}
& g_{({\bf i}, j)}a^{F(n)+R_n} a^{F(n+1)+R_{n+1}-F(n)-R_n}\overline{B}^U_{\eta_0/4}a^{-F(n+1)-R_{n+1}+F(n)+R_n} 
\\
& \qquad \subseteq g_{\bf i} a^{F(n)+R_n} \overline{B}^U_{\eta_0/4}. \nonumber
\end{align}
Since 
\[
 F(n+1)+R_{n+1} - F(n)-R_n = R_{n+1} + R'>0,
\]
we have
\[
 a^{F(n+1)+R_{n+1} - F(n)-R_n} \overline{B}^U_{\eta_0/4} a^{-F(n+1)-R_{n+1} + F(n) + R_n} \subseteq \overline{B}^U_{\lambda_0^{-1} \eta_0/4}.
\]
Then \eqref{nested} follows from
\[
 \lambda_0^{-1} \frac{\eta_0}4 + d\big(g_{({\bf i},j)}a^{F(n)+R_n}, g_{\bf i}a^{F(n)+R_n}\big) < \frac{\eta_0}4 \cdot \frac{1}{\lambda_0} + \frac{\eta_0}{4} \cdot \frac{\lambda_0-1}{\lambda_0} = \frac{\eta_0}4.
\]
Thus, the sets of the collection are nested in the required way. Finally, 
\[
 g a^{F(n)+R_n} \overline{B}^U_{\eta_0/4} a^{-F(n)-R_n} \subseteq g \overline{B}^U_{\lambda_0^{-F(n)-R_n} \eta_0/4},
\]
and hence
\[
 \diam\big( g a^{F(n)+R_n} \overline{B}^U_{\eta_0/4} a^{-F(n)-R_n}\big) \ll \lambda_0^{-F(n)-R_n}.
\]
Therefore, the sequence of supremal diameters converges to $0$ as $n\to \infty$. This completes the proof that $\mc U = \bigcup \mc U_n$ is a strongly tree-like collection.

Throughout we fix this choice of strongly tree-like collection. Moreover, we define the sets $\bf U_n$, $n\in\N_0$, and $\bf U_\infty$ as in \eqref{fatUn} and \eqref{fatUinfty}.

\begin{prop}\label{limitsetdivergent}
Let $x_1\in\compact = \homsp_{\leq s_0}$ be as in Theorem~\ref{thm:main}. Then $x_1g$ diverges on average for all $g\in \bf U_\infty$.
\end{prop}

\begin{proof}
The structure of the sets in $\mc U$ yields that $\bf U_\infty$ consists of the elements
\[
 g_\infty = \lim_{n\to\infty} g_{(i_1,\ldots, i_n)} = \bigcap_{n\in\N} g_{(i_1,\ldots, i_n)} a^{F(n)+R_n} \overline{B}^U_{\eta_0/4} a^{-F(n)-R_n},
\]
where $(i_k)_{k\in\N}$ is any sequence such that $i_k \in [1,S_k]$ for $k\in\N$. Let $\compact'$ be any compact subset of $\homsp$. Without loss of generality, we may assume that $\compact'= \homsp_{\leq s}$ for some large $s$. In the following we will prove that the amount of time (discrete time steps) in $[0,F(n)+R_n]$ which is spend in $\compact'$ by the points in 
\[
 x_1 g_{(i_1,\ldots, i_n)}a^{F(n)+R_n} \overline{B}^U_{\eta_0/4} a^{-F(n)-R_n}
\]
grows sublinear as $n\to\infty$. This will then prove the proposition. To start we remark that for any given point in $x\in\homsp$, its $T$-orbit $(xa^k)_{k\in\N_0}$ stays only a uniformly bounded number of consecutive steps in the strip $\homsp_{>s_1}\cap \homsp_{\leq s}$ (which is a due to the space $G/K$ being of rank one, see \cite{EKP}). Let
\begin{align*}
\ell & \sceq \max\{ k\in\N \mid \exists\, x\in\homsp_{\leq s_1}\colon Tx,\ldots, T^kx\in\homsp_{>s_1}\cap \homsp_{\leq s},\ T^{k+1}x\in\homsp_{>s} \}
\\
& = \max\{ k\in\N\mid \exists\, x\in\homsp_{>s}\colon Tx,\ldots, T^kx\in \homsp_{>s_1}\cap \homsp_{\leq s},\ T^{k+1}x\in\homsp_{\leq s_1}\}.
\end{align*}
By the choice of $s_1$, as soon as $\height(xa^k) > \height(xa^{k+1}) > s_1$, the orbit strictly descends until being below height level $s_1$. Since $s_0/39>s_1$, this means that as soon as the orbit stays for more than $2\ell$ consecutive steps above height $s_1$, say for $m$ steps, it necessarily stays  at least $m-2\ell$ steps in $\homsp_{>s}$. To simplify the proof we may assume that $s_0$ is chosen such that 
\[
 x\overline{B}^G_{\eta_0} \subseteq \homsp_{>s_1}
\]
for all $x\in\homsp_{>s_0/39}$. We use the notation of the proof of Theorem~\ref{thm:main}. Let $n\in\N$ and ${\bf i} = (i_1,\ldots, i_n)\in \mc S_n$. We claim that 
\begin{equation}\label{contained}
 x_1 g_{\bf i} a^{F(n) + R_n} \overline B^U_{\eta_0/4} a^{-F(n)-R_n + k} \subseteq x_m u_{i_m}^{(m)} a^{k-F(m)} \overline B^G_{\eta_0}
\end{equation}
for $k\in [F(m), F(m)+R_m]$ and $m=1,\ldots, n$. For $n=1$, this is clearly true. For ${\bf j}=(j_1,\ldots, j_{p+1}) \in \mc S_{p+1}$ for any $p\in\N$, the proof of Theorem~\ref{thm:main} showed the identities
\[
 g_{\bf j} = g_{(j_1,\ldots, j_p)} a^{F(p) + R_p} u^+_{j_{p+1}} a^{-F(p)-R_p}
\]
and
\[
 x_1 g_{\bf j} a^{F(p)+R_p} u^+_{j_{p+1}}a^{R'} = x_{p+1} u_{j_{p+1}}^{(p+1)} u_{j_{p+1}},
\]
where $u^+_{j_{p+1}}\in B^U_{c(c+2)\delta}$ and $u_{j_{p+1}}\in B^{NAM}_{c(c+2)\delta}$. For $m=1,\ldots, n-1$, these yield
\begin{align}\label{formg}
 \nonumber x_1 g_{\bf i} & = x_1 g_{(i_1,\ldots, i_m)} \prod_{p=0}^{n-m-1} a^{F(m+p)+R_{m+p}} u^+_{i_{m+p+1}} a^{-F(m+p)-R_{m+p}}
\\
 & = x_{m+1} u_{i_{m+1}}^{(m+1)} a^{-F(m+1)} \prod_{p=1}^{n-m-1} a^{F(m+p)+R_{m+p}} u^+_{i_{m+p+1}} a^{-F(m+p)-R_{m+p}}.
\end{align}
Therefore
\begin{align}\label{x11}
x_1 &g_{\bf i} a^{F(n)+R_n}\overline B^U_{\eta_0/4}a^{-F(n)-R_n+k}
\\
\nonumber &  = \left( x_{m+1}u_{i_{m+1}}^{(m+1)} a^{k-F(m+1)}\right)\left(a^{F(m+1)-k} u_{i_{m+1}} a^{-F(m+1)+k}\right)
\\
\nonumber & \times \prod_{p=1}^{n-m-1} \left(a^{F(m+p)+R_{m+p}-k} u^+_{i_{m+p+1}} a^{-F(m+p)-R-{m+p}+k}\right) 
\\
\nonumber & \times \left(a^{F(n)+R_n-k} \overline B^U_{\eta_0/4} a^{-F(n)-R_n+k}\right)
\end{align}
for $m=1,\ldots, n-1$, and
\begin{align}\label{x12}
x_1 &g_{\bf i} a^{F(n)+R_n} \overline B^U_{\eta_0/4} a^{-F(n)-R_n+k} 
\\
\nonumber & = x_1 g_{i_1} a^k \prod_{p=0}^{n-2} \left( a^{F(p+1)+R_{p+1} - k} u^+_{i_{p+2}} a^{-F(p+1)-R_{p+1} + k}\right) 
\\
\nonumber & \times\left( a^{F(n)+R_n-k}\overline B^U_{\eta_0/4} a^{-F(n)-R_n+k}\right).
\end{align}
For $k\in [F(m+1), F(m+1)+R_{m+1}]$, we have
\begin{align*}
\prod_{p=1}^{n-m-1} \left( a^{F(m+p)+R_{m+p}-k} u^+_{i_{m+p+1}} a^{-F(m+p)-R_{m+p}+k}\right) \in B^U_r
\end{align*}
with
\[
 r = c(c+2)\delta \sum_{p=1}^{n-m-1}\lambda_0^{-(F(m+p)+R_{m+p}-k)} \leq c(c+2)\delta \frac{1}{\lambda_0+1} \leq \frac{\eta_0}{4},
\]
and
\[
 a^{F(m+1)-k}u_{i_{m+1}}a^{-F(m+1)+k} \in B^{NAM}_{\eta_0/4}.
\]
Hence, \eqref{x11} implies \eqref{contained} for $2,\ldots, n$. By the same argumentation, \eqref{x12} implies \eqref{contained} for $1$ (note that $g_{i_1}=u_{i_1}^{(1)}$).

We consider \eqref{contained} for $m\in\{1,\ldots, n\}$ and $k\in [F(m), F(m)+R_m]$. Proposition~\ref{prop:main} shows that $x_mu_{i_m}^{(m)} a^{k-F(m)} \in \homsp_{>\frac{s_0}{39}}$, and hence $x_mu_{i_m}^{(m)}a^{k-F(m)}\overline B^G_{\eta_0}\subseteq \homsp_{>s_1}$ for all $k\in [F(m), F(m)+R_m]$. As discussed above, this implies that for any point $y\in x_1 g_{\bf i} a^{F(n)+R_n} \overline B^U_{\eta_0/4} a^{-F(n)-R_n}$, its $T$-orbit $(ya^k)_{k\in\N_0}$ stays above height $s$ for (at least) $k\in [F(m)+\ell, F(m)+R_m-\ell]$. Thus, in the time interval $[0, F(n)+R_n]$, this orbit stays above height $s$ for at least $\sum_{j=1}^n R_j - 2n\ell$ steps. In turn, $(ya^k)_{k\in\N_0}$ visits $\compact'$ for at most $(n-1)R'+2n\ell$ values for $k$ in $[0,F(n)+R_n]$. One easily sees that 
\[
 \lim_{n\to\infty} \frac{(n-1)R' + 2n\ell}{F(n) + R_n} =0,
\]
which completes the proof.
\end{proof}

\subsection{Hausdorff dimension}

\begin{prop}\label{HDU}
We have 
\[
 \dim_H {\bf U}_\infty \geq \frac{p_1}2 = \frac12 \dim U - \frac{p_2}2.
\]
\end{prop}

\begin{proof}
We apply Lemma~\ref{lem:Frostman}. Let $k\in\N$ and $B\in \mc U_k$. Then
\[
 \delta_{k+1}(B,\mc U) = \frac{\lambda({\bf U}_k\cap B)}{\lambda(B)} = \frac{S_{k+1}\cdot \lambda\big(a^{F(k+1)+R_{k+1}}\overline{B}^U_{\eta_0/4}a^{-F(k+1)-R_{k+1}}\big)}{\lambda\big(a^{F(k)+R_k}\overline{B}^U_{\eta_0/4}a^{-F(k)-R_k}\big)},
\]
and hence 
\[
 \Delta_k(\mc U) = \delta_{k+1}(B,\mc U).
\]
For any $L\in\N$ we have
\[
 \lambda\big(a^L \overline{B}^U_{\eta_0/4}a^{-L}\big) = \left(\frac{\eta_0}{2}\right)^{p_1+p_2} e^{-L\left(p_2 + \frac{p_1}2\right)} = \left(\frac{\eta_0}{2}\right)^{p_1+p_2} e^{-L h_m(T)}.
\]
Thus,
\[
 \Delta_k(\mc U) = S_{k+1} e^{-(R_{k+1}+R')h_m(T)}.
\]
Note that $R_{k+1} = k+k_0$ and 
\[
 e^{\frac12 R_{k+1} h_m(T)} \geq S_{k+1} = \left\lfloor e^{\frac{k+k_0}2} \right\rfloor^{p_2} \cdot \left\lfloor e^{\frac{k+k_0}4}\right\rfloor^{p_1} \geq e^{\frac{k}{2} h_m(T)}.
\]
Then 
\[
 1\geq c_2 e^{-\frac{k}{2}h_m(T)}\geq \Delta_k(\mc U) \geq c_1 e^{-\frac{k}{2} h_m(T)}
\]
for some constants $c_1,c_2$. It follows that
\begin{align*}
\sum_{k=1}^{n-1} \left|\log\big(\Delta_k(\mc U)\big)\right| \asymp \frac{h_m(T)}{2} \sum_{k=1}^{n-1}k \asymp \frac{h_m(t)}{4}n^2.
\end{align*}
Moreover
\[
 d_n(\mc U) \leq \frac{\eta_0}2 e^{-\frac12(F(n)+R_n)},
\]
and hence
\[
 \left|\log\big(d_n(\mc U)\big)\right| \geq c \frac{n^2}{4}
\]
for some constant $c$ and sufficiently large $n$. Then 
\begin{align*}
\limsup_{n\to\infty} \frac{\sum_{k=1}^{n-1} \left|\log\big(\Delta_k(\mc U)\big)\right|}{\left|\log \big(d_n(\mc U)\big)\right|} \leq  h_m(T).
\end{align*}
Since $\dim U = p_1 + p_2$, this completes the proof.
\end{proof}

\begin{proof}[Proof of Theorem~\ref{thm:lower}]
The space of possible $x$ in Proposition~\ref{prop:main} (and hence of possible $x_1$ in Theorem~\ref{thm:main} and Proposition~\ref{limitsetdivergent}) is at least of dimension $\dim(NAM)$. For the Hausdorff dimension of the set $\mc D$ of points in $\homsp$ which diverge on average this observation implies
\[
 \dim_H \mc D \geq \dim NAM + \dim {\bf U}_\infty.
\]
Now using Proposition~\ref{HDU} completes the proof.
\end{proof}

\section{Proof of Theorem~\ref{thm:entropylowerbound}}

In \cite{Const}, the first named author proved the corresponding statement of Theorem~\ref{thm:entropylowerbound} for $\SL_{d+1}(\Z)\backslash \SL_{d+1}(\R)$, $d\geq 1$, and the action of a certain (singular) diagonal element of $\SL_{d+1}(\R)$. For the proof he used the variational principle for entropy and established the existence of sufficiently large subsets of $(n,\eps)$-separated points in $\SL_{d+1}(\Z)\backslash\SL_{d+1}(\R)$ whose trajectories are bounded but stay high up (near the bound) for a significant ratio of time (see \cite[Theorem~3.2]{Const}). These subsets are necessarily adapted to $\SL_{d+1}(\Z)\backslash \SL_{d+1}(\R)$. In Proposition~\ref{thm:entropymain} below we show the analogous statement for $\Gamma\backslash G$ and $T$ being the time-one geodesic flow. After that, the proof of Theorem~\ref{thm:entropylowerbound} is an adaption of \cite{Const}. For the convenience of the reader, we provide some details.

\begin{prop}
\label{thm:entropymain}
Let $s>39s_1$. Then there exists $R'\in \N$ such that for all $R\in\N$, $R>4\log 4$, there is a subset $\wt E$ of $\homsp_{\leq s}$ such that the following properties are satisfied:
\begin{enumerate}[{\rm (i)}]
\item\label{maini} There exists $s'>s$ such that 
\[
 T^\ell x \in \homsp_{\leq s'}
\]
for all $x\in\wt E$ and all $\ell\in\N_0$.
\item For any $m\in\N$ we find a subset $\wt E(m)$ of $\wt E$ such that 
\begin{enumerate}[{\rm (1)}]
\item the cardinality of $\wt E(m)$ is $S^m$ with $S = S(R) =\lfloor e^{\frac{R}{4}}\rfloor^{p_1}\cdot\lfloor e^{\frac{R}{2}}\rfloor^{p_2}$,
\item\label{mainii2} $\wt E(m)$ is $(mR + (m-1)R', \eta')$-separated for some $\eta'> 0$ not depending on $m$, and
\item\label{mainii3} for any $x\in\wt E(m)$ we have
\[
 \left| \left\{ \ell \in [0,mR+(m-1)R'-1] \left\vert\ T^\ell x\in\homsp_{\geq \frac{s}{100}}\right.\right\}\right| \geq mR.
\]
\end{enumerate}
\end{enumerate}
\end{prop}

To prove Proposition~\ref{thm:entropymain} we need the following lemma, which is similar to Lemma~5.2 in \cite{Const}. We omit its proof. Let
\[
 \lambda_1 \sceq \max\{|\lambda| \mid \text{$\lambda$ is an eigenvalue of $\Ad_a$ with $|\lambda|>1$}\}.
\]
Thus,
\[
 \lambda_1 =
\begin{cases}
e^{1/2} & \text{if $\mf g_2 = \{0\}$ (and hence $G/K$ is real hyperbolic),}
\\
e & \text{otherwise.}
\end{cases}
\]

\begin{lemma}\label{lem:group}
Let $s'>0$ and pick an injectivity radius $\eta>0$ of $\mc X_{\le  s'}$. Let $n\in\N$ and suppose that $g,h\in U$ and $x_0\in\homsp$ are such that $T^\ell(x_0 g),T^\ell(x_0 h) \in \mc X_{\le s'}$ for all $\ell \in [0,n]$. Further suppose that $d(g,h)=d( x_0 g,  x_0 h)$ and that $d(T^n g,T^n h)>\frac{\eta}{\lambda_1}$. Then there exists $\ell \in [0,n]$ such that $d(T^\ell(x_0 g),T^\ell(x_0 h))\geq \frac{\eta}{\lambda_1}$.
\end{lemma}

\begin{proof}[Proof of Proposition~\ref{thm:entropymain}]
Let $\compact \sceq \homsp_{\leq s}$ and pick $\eta_0\in (0,1/2)$ such that it is an injectivity radius of $B_{\eta_0}(\compact)$. Apply Proposition~\ref{prop:main} with $\eta_0$ and $R$ to get a subset $E\subseteq B^U_{\eta_0/4}$ with 
\[
S=\lfloor e^{R/2} \rfloor^{p_2} \lfloor e^{R/4}\rfloor^{p_1}
\]
elements and $x\in\compact$ with properties as in that proposition. Let 
\[
 0< \eta < \frac{\eta_0 (\lambda_0-1)}{4\lambda_0}
\]
be small enough such that we may apply Theorem~\ref{thm:main}. In the following we will use the notation of Theorem~\ref{thm:main}. For $k\in\N$ define $R_k\sceq R$, $S_k\sceq S$, $E^{(k)}\sceq E$ and $x_k\sceq x$. Now Theorem~\ref{thm:main} provides $R'=R'(\eta,\compact)\in\N$ and a family of subsets
\[
 E'_n \sceq \{ g_{\bf i} \mid {\bf i}\in \mc S_n \},\quad n\in\N,
\]
of $U$ with the properties stated there. Let $\wt{\mc S} \sceq [1,S]^\N$ and let
\[
 {\bf i}_\infty = (i_j)_{j\in\N} \in \wt{\mc S}.
\]
As in the proof of Proposition~\ref{limitsetdivergent}, we see that $(g_{(i_1,\ldots, i_n)})_{n\in\N}$ is convergent. Let 
\[
 g_{{\bf i}_\infty} \sceq \lim_{n\to\infty} g_{(i_1,\ldots, i_n)}.
\]
Define
\[
 \wt E \sceq \left\{ x g_{{\bf i}_\infty} \left\vert\  {\bf i}_\infty \in \wt{\mc S}\right.\right\},
\]
and
\[
 \wt E(m) \sceq \left\{ xg_{{\bf i}_\infty} \left\vert\ {\bf i}_\infty \in\wt{\mc S},\ \text{$i_j=1$ for $j>m$}\right.\right\} \quad \text{for $m\in\N$.}
\]
Since the maximal variation of height under one application of $T$ is bounded, the sequence $(R_k)_k$ is constant (namely, $R$) and the starting points $xu$, $u\in E$, are contained in a compact set, we deduce from \eqref{formg} in the proof of Proposition~\ref{limitsetdivergent} (and a limit over $n$) that we find $s'>s$ such that the $T$-orbit of each element in $\wt E$ is contained in the compact set $\homsp_{\leq s'}$. 

Let $n\in\N$, ${\bf i}\in \mc S_n$ and $m\in\{1,\ldots, n\}$. From \eqref{formg} it follows that
\[
 xg_{\bf i} a^k \in xu_j a^{k-F(m)} \overline B^U_{\eta/2}
\]
for some $j\in \{1,\ldots, S\}$ and all $k\in [F(m), F(m)+R]$. Since $xu_ja^{k-F(m)} \in \homsp_{\geq s/39}$, we have $xg_{\bf i} a^k \in \homsp_{\geq \frac{s}{39} - \frac{\eta}{2}}$. Note that $\eta$ does not depend on $n,m$ or ${\bf i}$. Thus, for any $x\in \wt E$ it follows that 
\[
  \left| \left\{ \ell\in [0,mR+(m-1)R'-1] \left\vert\  T^\ell x \in \homsp_{\geq \frac{s}{39}+ \frac{\eta}{2}}\right.\right\}\right| \geq mR.
\]
For $\eta$ sufficiently small, this proves \eqref{mainii3}.

Obviously, the cardinality of $\wt E(m)$ is at most $S^m$. The equality follows from \eqref{mainii2}. For the proof of \eqref{mainii2} we want to make use of Lemma~\ref{lem:group}. For ${\bf i}_\infty,{\bf j}_\infty\in \wt{\mc S}$, Theorem~\ref{thm:main} yields $d(g_{{\bf i}_\infty},g_{{\bf j}_\infty})<\eta_0$. 
The proof of Proposition~\ref{limitsetdivergent} shows
\[
 xg_{\bf i} \in xg_{i_1} B^U_{\eta_0/4}
\]
for each ${\bf i} = (i_1,\ldots, i_n)\in \mc S_n$, $n\in\N$. It follows that $xg_{{\bf i}_\infty}, xg_{{\bf j}_\infty}\in B_{\eta_0}(\compact)$. Then $\eta_0$ being an injectivity radius of $B_{\eta_0}(\compact)$ yields
\[
 d(g_{{\bf i}_\infty},{{\bf j}_\infty}) = d(xg_{{\bf i}_\infty},xg_{{\bf j}_\infty}).
\]
Now let $m\in\N$ and ${\bf i} = (i_1,\ldots, i_m), {\bf j}=(j_1,\ldots, j_m)\in \mc S_m$, ${\bf i}\not= {\bf j}$. We claim that 
\[
 d\big( T^{F(m)+R}g_{({\bf i}, {\bf 1})}, T^{F(m)+R}g_{({\bf j},{\bf 1})}\big) > \frac{\eta_0}4,
\]
where $({\bf i}, {\bf 1})$ denotes the element in $\wt{\mc S}$ which extends ${\bf i}$ with $1$'s. We have
\begin{align*}
d\big( g_{\bf i}a^{F(m)+R},& g_{\bf j}a^{F(m)+R}\big)  \leq d\big( g_{\bf i}a^{F(m)+R}, g_{({\bf i}, {\bf 1})}a^{F(m)+R}\big) 
\\
& + d\big( g_{({\bf i},{\bf 1})}a^{F(m)+R}, g_{({\bf j},{\bf 1})}a^{F(m)+R}\big) + d\big(g_{({\bf j},{\bf 1})}a^{F(m)+R}, g_{\bf j}a^{F(m)+R}\big).
\end{align*}
By Theorem~\ref{thm:main}\eqref{Eniv}, 
\[
 d\big( g_{\bf i}a^{F(m)+R}, g_{\bf j}a^{F(m)+R}\big) > \frac{\eta_0}2.
\]
Let ${\bf 1}_n\sceq (1,\ldots, 1)\in \mc S_n$. Then
\[
 d\big(g_{\bf i}a^{F(m)+R}, g_{({\bf i}, {\bf 1})}a^{F(m)+R}\big) = \lim_{n\to\infty} d\big(g_{\bf i}a^{F(m)+R}, g_{({\bf i}, {\bf 1}_n)}a^{F(m)+R}\big)
\]
Since (see the proof of Proposition~\ref{limitsetdivergent})
\[
 g_{({\bf i}, {\bf 1}_n)} = g_{\bf i} \prod_{p=0}^{n-1} a^{F(m+p)+R} u^+_{i_{m+p+1}} a^{-F(m+p)-R}
\]
we find
\begin{align*}
d\big(g_{\bf i} a^{F(m)+R}, g_{({\bf i}, {\bf 1})}a^{F(m)+R}\big) & = 
\lim_{n\to\infty} d\Big(1, \prod_{p=0}^{n-1} a^{F(m+p) - F(m)} u^+_{i_{m+p+1}} a^{-F(m+p)+F(m)}\Big)
\\
& = \lim_{n\to\infty} d\Big( 1, \prod_{p=0}^{n-1} a^{p(R+R')} u^+_{i_{m+p+1}} a^{-p(R+R')}\Big)
\\
& \leq c(c+2)\delta \sum_{p=0}^\infty \lambda_0^{-p(R+R')}
\\
& < \frac{\eta_0}{8} \frac{(\lambda_0-1)^2}{\lambda_0^2} \frac{1}{1-\lambda_0^{-(R+R')}} < \frac{\eta_0}8.
\end{align*}
From this the claim follows. Pick now an injectivity radius $\eta'$ of $\homsp_{\leq s'}$ such that $\eta_0/4 \geq \eta'$. Applying Lemma~\ref{lem:group} with $\eta'$ completes the proof.
\end{proof}

\begin{lemma}\label{specialmeasure}
For any $\eps>0$ and any $s>s_1$ there exists a $T$-invariant probability measure $\mu$ on $\homsp$ such that 
\[
 h_\mu(T) > \frac12 h_m(T) -\eps \quad\text{and}\quad  \mu(\homsp_{\geq s}) > 1-\eps.
\]
\end{lemma}

\begin{proof}
Throughout we use the notation of Proposition~\ref{thm:entropymain}. We apply this proposition with $100s$ to obtain the constant $R'\in\N$. We pick $R\in\N$, $R>4\log 4$, such that 
\[
 \frac{R}{R+R'} > 1-\eps \quad\text{and}\quad \frac{\log S(R)}{R+R'} > \frac12 h_m(T) - \eps.
\]
Note that this choice is possible since 
\begin{align*}
S(R) & = \left\lfloor e^{\frac{R}{4}}\right\rfloor \cdot \left\lfloor e^{\frac{R}2}\right\rfloor^{p_2} 
 > \left( e^{\frac{R}{4}} - 1\right)^{p_1} \cdot \left( e^{\frac{R}{2}} - 1\right)^{p_2} 
\\
& \to e^{\frac{R}{2}h_m(T)} \quad\text{as $R\to\infty$.}
\end{align*}
Now we choose a subset $\wt E$ of $\homsp_{\leq 100s}$ and a family $(\wt E(m))_{m\in\N}$ of subsets of $\wt E$ with the properties as in Proposition~\ref{thm:entropymain}. For $m\in\N$ let $\sigma_m$ denote the uniform probability measure on $\wt E(m)$, that is,
\[
 \sigma_m \sceq \frac{1}{S^m} \sum_{x\in\wt E(m)} \delta_x,
\]
where $\delta_x$ denotes the Dirac measure with support $\{x\}$. Finite averaging of $\sigma_m$ provides us with the probability measures 
\[
 \mu_m\sceq \frac{1}{mR + (m-1)R'} \sum_{i=0}^{mR + (m-1)R'-1} T^i_*\sigma_m
\]
on $\homsp$ with support 
\[
 \bigcup_{i=0}^{mR+(m-1)R'-1} T^i\wt E(m) \quad\subseteq \bigcup_{i\in\N_0} T^i\wt E \seqc \mc E.
\]
By Proposition~\ref{thm:entropymain}\eqref{maini} we find $s'>100s$ such that $\mc E \subseteq \homsp_{\leq s'}$. Let $\mu$ be any weak* limit of $(\mu_m)_{m\in\N}$. Then $\mu$ is $T$-invariant and, due to the compactness of $\homsp_{\leq s'}$, a probability measure. Note that 
\[
 \mc K \sceq \bigcap_{j\in\N_0} T^{-j}\homsp_{\leq s'}
\]
is a compact subset of $\homsp$ on which $T$ induces an action, and $\mc E\subseteq \mc K$. Thus, $\mu$ can be considered as a $T$-invariant probability measure on $\mc K$. Since each set $\wt E(m)$, $m\in\N$, is $(mR + (m-1)R',\eta')$-separated, respectively, the proof of the Variational Principle \cite[Theorem~8.6]{Walters} shows 
\[
 h_\mu(T) \geq \liminf_{m\to\infty} \frac{\log S^m}{mR + (m-1)R'} = \frac{\log S}{R+R'}.
\]
By the choice of $R$, we have
\[
 h_\mu(T) > \frac12 h_m(T) - \eps.
\]
Moreover, Proposition~\ref{thm:entropymain}\eqref{mainii3} and the choice of $R$ give
\[
 \mu_m(\homsp_{\geq s}) \geq \frac{mR}{mR+(m-1)R'} > \frac{R}{R+R'} > 1-\eps.
\]
Thus, 
\begin{align*}
\mu(\homsp_{\geq s}) & = \mu(\mc K \cap \homsp_{\geq s}) = \lim_{m\to\infty} \mu_m(\mc K \cap \homsp_{\geq s}) = \lim_{m\to\infty} \mu_m(\homsp_{\geq s}) > 1-\eps.
\end{align*}
This proves the lemma.
\end{proof}

For the proof of Theorem~\ref{thm:entropylowerbound} we recall that $m$ denotes the normalized Haar measure on $\homsp$.

\begin{proof}[Proof of Theorem~\ref{thm:entropylowerbound}]
For sufficiently large $n\in\N$ we apply Lemma~\ref{specialmeasure} with $\eps = \frac1n$ and $s=n$ to obtain a $T$-invariant probability measure $\mu_n$ on $\homsp$ with $\mu_n(\homsp_{\geq n}) > 1-\frac1n$ and 
\begin{equation}\label{boundentropsm}
 h_{\mu_n}(T) > \frac12 h_m(T) - \frac1n.
\end{equation}
Then the weak* limit of the sequence $(\mu_n)_n$ is the zero measure. Now \eqref{boundentropsm} and \cite[Theorem~7.5]{EKP} (the theorem presented in the Introduction) show
\[
 \lim_{n\to\infty} h_{\mu_n}(T) = \frac12 h_m(T).
\]
Thus, Theorem~\ref{thm:entropylowerbound} is proven for the case $c=\frac12h_m(T)$. If $c$ is any value in the interval $[\frac12 h_m(T), h_m(T)]$, then we consider the sequence $(\nu_n)_n$ of $T$-invariant probability measures on $\homsp$ given by the convex combination
\[
 \nu_n \sceq \left( \frac{2c}{h_m(T)} - 1 \right) m + \left( 2 - \frac{2c}{h_m(T)} \right) \mu_n.
\]
Recall that $m$ denotes the normalized Haar measure on $\homsp$. 
Its weak* limit $\nu$ satisfies 
\[
 \nu = \lim_{n\to\infty} \nu_n = \left( \frac{2c}{h_m(T)} - 1 \right) m,
\]
hence
\[
 \nu(\homsp) = \frac{2c}{h_m(T)} - 1.
\]
Moreover, 
\begin{align*}
\lim_{n\to\infty} h_{\nu_n}(T) & = \left( \frac{2c}{h_m(T)} - 1 \right) h_m(T) + \left( 2 - \frac{2c}{h_m(T)} \right) \lim_{n\to\infty} h_{\mu_n}(T)
\\
& = c.
\end{align*}
This finishes the proof.
\end{proof}

\bibliographystyle{amsalpha}
\bibliography{ap_bib}

\providecommand{\bysame}{\leavevmode\hbox to3em{\hrulefill}\thinspace}
\providecommand{\MR}{\relax\ifhmode\unskip\space\fi MR }
\providecommand{\MRhref}[2]{%
  \href{http://www.ams.org/mathscinet-getitem?mr=#1}{#2}
}
\providecommand{\href}[2]{#2}
\begin{thebibliography}{CDKR98}

\bibitem[CDKR91]{CDKR1}
M.~Cowling, A.~Dooley, A.~Kor{\'a}nyi, and F.~Ricci, \emph{{$H$}-type groups
  and {I}wasawa decompositions}, Adv. Math. \textbf{87} (1991), no.~1, 1--41.

\bibitem[CDKR98]{CDKR2}
\bysame, \emph{An approach to symmetric spaces of rank one via groups of
  {H}eisenberg type}, J. Geom. Anal. \textbf{8} (1998), no.~2, 199--237.

\bibitem[Dan85]{DaniDivergent}
S.~G. Dani, \emph{Divergent trajectories of flows on homogeneous spaces and
  {D}iophantine approximation}, J. Reine Angew. Math. \textbf{359} (1985),
  55--89.

\bibitem[EK]{Einsiedler_Kadyrov}
M.~Einsiedler and S.~Kadyrov, \emph{Entropy and escape of mass for
  $\text{SL}_3(\mathbb{Z})\backslash \text{SL}_3(\mathbb{R})$},
  arXiv.org:0912.0475.

\bibitem[EKP]{EKP}
M.~Einsiedler, S.~Kadyrov, and A.~Pohl, \emph{Escape of mass and entropy for
  diagonal flows in real rank one situations}, arXiv:1110.0910v1.

\bibitem[Fal86]{Fal}
K.~J. Falconer, \emph{The geometry of fractal sets}, Cambridge Tracts in Math.,
  vol. 85, Cambridge Univ. Press, Cambridge and New York, 1986.

\bibitem[GR70]{Garland_Raghunathan}
H.~Garland and M.~S. Raghunathan, \emph{Fundamental domains for lattices in
  ({R}-)rank {$1$} semisimple {L}ie groups}, Ann. of Math. (2) \textbf{92}
  (1970), 279--326.

\bibitem[Kad12]{Const}
S.~Kadyrov, \emph{Positive entropy invariant measures on the space of lattices
  with escape of mass}, Ergodic Theory and Dynamical Systems \textbf{32}
  (2012), 141--157.

\bibitem[KM96]{KleMar}
D.~Y. Kleinbock and G.~A. Margulis, \emph{Bounded orbits of nonquasiunipotent
  flows on homogeneous spaces}, Amer. Math. Soc. Translations \textbf{171}
  (1996), 141--172.

\bibitem[McM87]{McM}
C.~McMullen, \emph{Area and {H}ausdorff dimension of {J}ulia sets of entire
  functions}, Trans. Amer. Math. Soc. \textbf{300} (1987), 329--342.

\bibitem[Poh10]{Pohl_isofunddom}
A.~Pohl, \emph{Ford fundamental domains in symmetric spaces of rank one}, Geom.
  Dedicata \textbf{147} (2010), 219--276.

\bibitem[PW94]{PesWei}
Ya. Pesin and H.~Weiss, \emph{On the dimension of deterministic and random
  {C}antor-like sets}, Math. Res. Lett. \textbf{1} (1994), 519--529.

\bibitem[Urb91]{Urb}
M.~Urbanski, \emph{The {H}ausdorff dimension of the set of points with nondense
  orbit under a hyperbolic dynamical system}, Nonlinearity \textbf{2} (1991),
  385--397.

\bibitem[Wal00]{Walters}
Peter Walters, \emph{An introduction to ergodic theory}, Graduate Texts in
  Mathematics, vol.~79, Springer-Verlag, New York, 2000.

\end{thebibliography}

\end{document}